\documentclass{amsart}
\usepackage{amssymb,amsfonts}
\usepackage{latexsym}
\usepackage{amsmath,mathrsfs}
\usepackage[all,arc]{xy}
\usepackage{enumerate}
\usepackage[english]{babel}
\usepackage[bookmarks=true]{hyperref}
\usepackage{tikz}
\usepackage{blkarray}
\usepackage{graphicx}
\newtheorem{theo}{Theorem}[section]

\newtheorem{prop}[theo]{Proposition}
\newtheorem{lemm}[theo]{Lemma}

\newtheorem*{claim}{Claim}

\theoremstyle{definition}
\newtheorem{defi}[theo]{Definition}

\theoremstyle{remark}
\newtheorem{rema}[theo]{Remark}

\newcommand{\bb}[1]{\mathbb{#1}}
\newcommand{\al}[1]{\mathcal{#1}}
\newcommand{\sr}[1]{\mathscr{#1}}
\newcommand{\ak}[1]{\mathfrak{#1}}
\newcommand{\gen}[1]{\left\langle #1\right\rangle}
\newcommand{\sm}{\smallsetminus}

\newcommand{\bs}{\backslash}
\newcommand{\ra}{\rightarrow}
\newcommand{\lra}{\longrightarrow}
\newcommand{\x}[1]{\text{#1}}
\newcommand{\xb}[1]{\text{\textbf{#1}}}

\begin{document}
\title[Anti-invariant  bundles and twisted conformal blocks]{Moduli spaces of anti-invariant vector bundles and twisted conformal blocks}
\author{Hacen ZELACI}
\address{Mathematical Institute of the University of Bonn}
\curraddr{Endenicher Allee 60, Bonn, Germany.}
\email{z.hacen@gmail.com}
\date{\today}

\begin{abstract}
We prove a canonical identifications between the spaces of generalized theta functions on the moduli spaces of anti-invariant  vector bundles in the ramified case and the conformal blocks associated to twisted Kac-Moody affine algebras. We also show a strange duality on level one in the unramiffied case, this gives the dimensions of the spaces of generalized theta functions of level one.
\end{abstract}
\maketitle
\tableofcontents

\section{Introduction}
Consider a smooth projective complex curve $X$ of genus $g_X\geqslant2$ with an involution $\sigma$ and assume that the fixed locus of $\sigma$ is not empty and contains $2n$ points. Let $\pi:X\ra Y=X/\sigma$ the associated double cover and denote $R$ the ramification locus and $B=\pi(R)$.\\ 
 An anti-invariant  vector bundle over $X$ is a vector bundle $E$ that has an isomorphism $\psi:\sigma^*E\ra E^*$. If the this isomorphism verifies $\sigma^*\psi=\,^t\psi$ then $E$ is called $\sigma-$symmetric, and  if it verifies  $\sigma^*\psi=-\,^t\psi$ then  it is called $\sigma-$alternating, in this case the rank has to be  even. If $E$ is stable then $\psi$ is necessarily $\sigma-$symmetric or $\sigma-$alternating. \\
 We constructed the moduli spaces of such vector bundles in \cite{Z2}, and in \cite{Z}, we showed that the locus $\al {SU}_X^{\sigma,+}(r)$ of stable $\sigma-$symmetric anti-invariant vector bundles with trivial determinant over $X$ is irreducible.  We also proved that  the locus $\al{SU}_X^{\sigma,-}(r)$ (for even rank $r$) of stable $\sigma-$alternating vector bundles with trivial determinant has $2^{2n-1}$ connected components indexed by some types $\tau=(\tau_p)_{p\in R}\mod \pm 1$, where $\tau_p\in \{\pm1\}$ is the \emph{Pfaffian} of the $\sigma-$alternating isomorphism $\psi:\sigma^*E\ra E^*$ over $p$.  

The main topic of this paper is the study of generalized theta functions on the moduli stacks $\sr{SU}_X^{\sigma,+}(r)$ (resp. $\sr{SU}_X^{\sigma,\tau}(r)$) of $\sigma-$symmetric (resp. $\sigma-$alternating of type $\tau$) anti-invariant vector bundles $(E,\delta,\psi)$ over $X$, where $\delta$ is a trivialisation of $\x{det}(E)$ and $\psi:\sigma^*E\cong E^*$ is $\sigma-$symmetric (resp. $\sigma-$alternating of type $\tau$) compatible with $\delta$. 
When the type is $\tau$ trivial, that's $\tau_i=+1$ for all $i$, the corresponding moduli stack is denoted simply by $\sr{SU}_X^{\sigma,-}(r)$.

These moduli stacks correspond to moduli stacks  $\sr M_Y(\al G)$ of parahoric $\al G-$torsors over the quotient curve $Y$, for some \emph{twisted} parahoric Bruhat-Tits  group schemes $\al G$. Parahoric  $\al G-$torsors have attracted the attention of many mathematician recently (see \cite{PR}, \cite{He}, \cite{BS}, \cite{BKV}), since they can be considered as a generalization of parabolic $G-$bundles. Their moduli spaces has been constructed in the untwisted case by Balaji and Seshadri \cite{BS}. For the twisted case, we have constructed their moduli spaces in type $A_n$ in \cite{Z2}.\\ 

The restriction of the determinant bundle $\al D$ over the moduli stack of vector bundles of trivial determinant $\sr{SU}_X(r)$ to the moduli stack $\sr{SU}_X^{\sigma,-}(r)$  turns out to have a square root for each $\sigma-$invariant theta characteristic, we call them  the \emph{Pfaffian of cohomology line bundles}. However this is not true for the $\sigma-$symmetric case. \\
In the \'etale case, the two stacks are isomorphic (see \cite{Z}). In this case also, the determinant bundle has a square root. However, in this paper we stick to the ramified case.\\

Our main result is the identification of the space of global sections of the powers of a Pfaffian line bundle $\al P$ and the determinant line bundle $\al D$ (called space of generalized theta functions) with the \emph{conformal blocks} $\al V_{\sigma, \pm}(k)$ associated to the twisted universal central extension of $\ak{sl}_r$ (see section \ref{sec2} for a precise definition).
\begin{theo} \label{main1}
	Let $\al P$ be a Pfaffian line bundle over $\sr{SU}_X^{\sigma,-}(r)$ and $\al D$ be the determinant line bundle over $\sr{SU}_X^{\sigma,+}$, and let $k\in \bb N$. Then we have canonical isomorphisms $$H^0(\sr{SU}_X^{\sigma,-}(r), \al P^k)\cong \al V_{\sigma,-}(k),$$  $$H^0(\sr{SU}_X^{\sigma,+}(r), \al D^k)\cong \al V_{\sigma,+}(k).$$  \\
\end{theo}

Using the results of Heinloth \cite{He}, describe the Picard group of this moduli stack. In fact it is infinite cyclic group whose generator is the Pfaffian bundle.
\begin{theo}
	Let $\al P$ be the Pfaffian bundle over $\sr{SU}_X^{\sigma,-}(r)$, then $\x{Pic}(\sr{SU}_X^{\sigma,+}(r))=\bb Z\al P.$
\end{theo}

\subsection*{Acknowledgement}
I would like to thank Christian PAULY for the useful discussions and for his continuous support. This work  has been established during the preparation of my thesis at J.A. Dieudonn\'e Laboratory, University of Nice-Sophia Antipolis. \\

\section{Preliminaries on twisted Kac-Moody algebras} \label{sec2}
In this first section, we recall briefly the construction of the  twisted affine Kac-Moody Lie algebras and the attached conformal blocks. We use notations of \cite{K}. The definition of twisted conformal blocks is adapted from  \cite{FS}, where a more general definition is given in the framework of vertex algebras.

Consider an outer automorphism $\tau$ of the Lie algebra $\ak{sl}_r(\bb C)$. It is an order two automorphism. The involution $\tau$ is extended to an automorphism of the affine Kac-Moody algebra $\widehat{\al L}(\ak{sl}_r)=\ak{sl}_r(\sr K)\oplus \bb CK$, where $\sr K=\bb C((t))$ and $K$ a central element,  by sending $x\otimes g(t)$ to $\tau(x)\otimes g(-t)$ and fixes the center. Then the fixed subalgebra of this involution, denoted by $\widehat{\al L}(\ak{sl}_r,\tau)$, is an affine Lie algebra of type $A^{(2)}_{l}$ (after adding a scaling element $D$), where $l=\lfloor r/2\rfloor $, and it is called \emph{twisted} affine Lie algebra. 
Let $\mathring{\ak g}$ be the finite simple Lie algebra of $\al{L}(\ak{sl}_r,\tau)$ (see \cite[\S6.3]{K} for a precise definition). Then $\mathring{\ak g}$ is of type  $C_{l}$ if $r$ is odd, and is isomorphic to the fixed subalgebra  $\ak{sl}_r(\bb C)^{\tau}$ if $r$ is even.

Since we will be interested mainly in the following two involutions $$\sigma^+(a(t))=-\,^ta(-t), \;\; \sigma^-(a(t))=-J_r\,^ta(-t)J_r^{-1},$$ where $$J_r=\begin{pmatrix}
0 & I_{r/2} \\ -I_{r/2} & 0
\end{pmatrix},$$ we give an explicit constructions  of $\al{L}(\ak{sl}_r,\sigma^\pm)$. 

Let $M_{i,j}$ be the canonical basis of the vector space of square matrices of size $r$. Let $\ak{h}\subset\ak{sl}_r(\bb C)$ be the Cartan subalgebra  of diagonal matrices  and let $\alpha'_1,\cdots,\alpha'_{r-1}\in \ak{h}^*$ be the simple roots defined by $\alpha_i=M_{i,i}^*-M_{i+1,i+1}^*$. Denote by  $E'_1,\cdots,E'_{r-1}$ and $F'_1,\cdots,F'_{r-1}$  the Chevalley generators of $\ak{sl}_r(\bb C)$: $E'_i=M_{i,i+1}$, $F'_i=-\,^tE'_i$.  Let  ${\alpha'_0}^\vee=M_{1,1}-M_{r,r}$, $E'_0=M_{1,r}$ and $F'_0=-\,^tE'_0$. Then the Chevalley generators of $\widehat{\al{L}}(\ak{sl}_r)$ are given by $$e_0=t\otimes E'_0,\;\; f_0=t^{-1}\otimes F'_0,$$ and for $i\in\{1,\cdots,r-1\}$ $$e_i=1\otimes E'_i,\;\; f_i=1\otimes F'_i. $$ 
Recall the Lie bracket on $\hat{\al L}(\ak{sl}_r)$ is given by 
$$[g(t),h(t)]=[g,h]\otimes P(t)Q(t)+(g,h)\x{Res}(\frac{dP}{dt}Q)K,$$
where $g,h\in \ak{sl}_r$, $P,Q\in\sr K$ and $(\,,\,)$ is the normalized Killing form on $\ak{sl}_r$.\\
Moreover, by extending the linear forms $\alpha_i$ to $\ak{h}\oplus\bb CK$ such that $\alpha_i(K)=0$, then $\alpha_i$ are the simple roots of $\widehat{\al{L}}(\ak{sl}_r)$.\\

{\bfseries Case $\al L(\ak{sl}_r,\sigma^-)$.} Let $r=2l$.  This is the algebra constructed in \cite[Page $128$]{K}.  We can assume, after conjugation, that $\sigma^-$ sends $E'_i$ to $E'_{r-i}$, $F'_i\ra F'_{r-i}$ and  $\alpha'_i\ra \alpha'_{r-i}$.  So let's define 
\begin{itemize} 	
	\item $\alpha_i^\vee= {\alpha'}_i^\vee+{\alpha'}_{r-i}^\vee$ ($1\leqslant i \leqslant l-1$), $\alpha^\vee_l={\alpha'}_l^\vee$ and $\alpha^\vee_0=-2{\alpha'}_0^\vee+{\alpha'}_1^\vee + {\alpha'}_{r-1}^\vee$. 
	\item $E_i=E'_i+E'_{r-i}$ ($1\leqslant i \leqslant l-1$), $E_l=E'_l$ and $E_0=E'_{-\alpha'_0+\alpha'_{r-1}}-E'_{-\alpha'_0+\alpha'_{1}}$. 
	\item $F_i=F'_i+F'_{r-i}$ ($1\leqslant i \leqslant l-1$), $F_l=F'_l$ and $F_0=-E'_{\alpha'_0-\alpha'_{r-1}}+E'_{\alpha'_0-\alpha'_1}$.    
\end{itemize} 
The Chevalley generators of $\widehat{\al L}(\ak{sl}_r,\sigma^-)$ are given by $$e_i=1\otimes E_i,\;\;f_i=1\otimes F_i\;\;\;\x{for }i=1,\dots,l. $$ $$e_0=t\otimes E_0,\;\; f_0=t^{-1}\otimes F_0. $$
Consider the elements  $\tilde{\alpha}_i^\vee=2\alpha_i^\vee/(\alpha_i^\vee,\alpha_i^\vee)\in \ak{h}$. Since the normalized bilinear form $(\,;\,)$ is non-degenerate on $\ak{h}$ it induces an isomorphism $\ak{h}\cong \ak{h}^*$. So let $\tilde{\alpha}_i$ be the images of $\tilde{\alpha}_i^\vee$ under this bijection.  Then the simple roots  of $\widehat{\al L}(\ak{sl}_r,\sigma^-)$ are given by $$\alpha_0=\frac{1}{2}\otimes \tilde{\alpha}_0, $$ $$\alpha_i=1\otimes \tilde{\alpha}_i,\;\; i=1,\dots,l.$$
The simple  coroots  are just $1\otimes\alpha_i^\vee$, for   $i=1,\dots,l.$ We denote them again by $\alpha_i^\vee.$ For $i=0$ the simple coroot is $2K+1\otimes \alpha_0^\vee.$ We denote it also by $\alpha_0^\vee$. \\
In particular, the normalized bilinear form on $\widehat{\al L}(\ak{sl}_r,\sigma^-)$ is given by $$(P\otimes x;Q\otimes y)=\dfrac{1}{2}\x{Res}(t^{-1}PQ)(x;y),$$
where $(\;,\;)$ is the normalized bilinear form on $\ak{sl}_r(\bb C)$. The $2-$cocycle on $\al{L}(\ak{sl}_r,\sigma^-)$ that defines $\widehat{\al L}(\ak{sl}_r,\sigma^-)$ is given by $$\psi(g(t),h(t))=\dfrac{1}{2}\x{Res}(\x{Tr}(\dfrac{dg}{dt}h)).$$ \\

{\bfseries Case $\al L(\ak{sl}_r,\sigma^+)$.} We treat the case  $r=2l$ (the odd case is again treated in \cite{K}).  We can assume, after conjugation, that $\sigma^+$ sends $E'_i$ to $-E'_{r-i}$, $F'_i\ra -F'_{r-i}$ and $\alpha'_i\ra \alpha'_{r-i}$. So we define the following elements of $\ak{sl}_{2l}$ :
\begin{itemize} 	
	\item $\beta_i^\vee= {\alpha'}_{l-i}^\vee+{\alpha'}_{l+i}^\vee$ ($1\leqslant i \leqslant l-1$), $\beta^\vee_l={\alpha'}_0^\vee$ and $\beta^\vee_0=2{\alpha'}_{l}^\vee+{\alpha'}_{l-1}^\vee+{\alpha'}_{l+1}^\vee$. 
	\item $E_i=E'_i-E'_{r-i}$ ($1\leqslant i \leqslant l-1$), $E_l=E'_0$ and $E_0=E'_{\alpha'_l+\alpha'_{l+1}}-E'_{\alpha'_{l}+\alpha'_{l-1}}$. 
	\item $F_i=F'_i-F'_{r-i}$ ($1\leqslant i \leqslant l-1$), $F_l=F'_0$ and $F_0=-E'_{-\alpha'_{l}-\alpha'_{l+1}}+E'_{-\alpha'_{l}-\alpha'_{l-1}}$.    
\end{itemize} 
Remark that the affine node $\beta_0$ of $\widehat{\al{L}}(\ak{sl}_r,\sigma^+)$ is then the node $\alpha_l$ with the notation of Table Aff\,2 of \cite[Page $55$]{K}. Thus when deleting this node the remaining diagram is of type $D_l$.\\  

As before, we define the Chevalley generators of $\widehat{\al L}(\ak{sl}_r,\sigma^+)$  by $$e_i=1\otimes E_i,\;\;f_i=1\otimes F_i\;\;\;\x{for }i=1,\dots,l. $$ $$e_0=t\otimes E_0,\;\; f_0=t^{-1}\otimes F_0. $$
The simple coroots of the simple invariant Lie algebra ($=\ak{so}_{2l}$) are given  by $$\tilde{\beta}_i^\vee=2\beta_i^\vee/(\beta_i^\vee,\beta_i^\vee),\;\; i=0,\dots,l.$$ As above denote by $\tilde{\beta}_i$ the corresponding elements of $\ak{h}^*$. Then the  simple roots of $\widehat{\al{L}}(\ak{sl}_{2l}, \sigma^+)$ are given by $$\beta_0=2K+1\otimes\tilde{\beta}_0,$$ $$\beta_i=1\otimes \tilde{\beta}_i, \;\;i=1,\dots,l.$$

From the construction of $\widehat{\al L}(\ak{sl}_r,\sigma^+)$, it is clear that the Coxeter coefficients and their duals in this case are taken in the inverse order. We recall  the dual Coxeter coefficients of the twisted Kac-Moody algebras $\widehat{\al L}(\ak{sl}_r,\sigma^\pm)$ in the following table. 

\begin{equation}\label{table}
	\begin{array}{|l|c|c|c|c|c|c|} 
		\cline{2-7}
		\multicolumn{1}{l|}{}& a_0^\vee & a_1^\vee & a_2^\vee & \cdots &  a_{l-1}^\vee & a_l^\vee\\\hline
		\widehat{\al L}(\ak{sl}_{2l},\sigma^+) & 2 &2 &2 & \cdots  & 1 & 1 \\\hline
		\widehat{\al{L}}(\ak{sl}_{2l},\sigma^-) & 1 & 1 & 2 & \cdots  & 2 & 2 \\\hline
		\widehat{\al L}(\ak{sl}_{2l+1},\sigma^+) & 1 & 2 & 2 & \cdots  & 2 & 2 \\
		\hline
	\end{array}
\end{equation}
\begin{center}
	Dual Coxeter  coefficients.\\[1cm]
\end{center}

Now, when we add a scaling elements to the above algebras, i.e.  derivations $D_\pm$ such that $$[D_\pm,t^n\otimes x]=nt^n\otimes x,$$
then, by \cite[Theorem $8.5$]{K}, both Kac-Moody algebras $\hat{\al L}(\ak{sl}_r,\sigma^\pm)\oplus\bb C D_\pm$ are isomorphic to the Kac-Moody algebra $\ak{g}(A)$, where $A$ is the affine generalized Cartan  matrix of type $A^{(2)}_{r-1}$. In particular, we deduce an isomorphism $$\widehat{\al{L}}(\ak{sl}_r,\sigma^+)\oplus\bb CD_+\cong \hat{\al{L}}(\ak{sl}_r,\sigma^-)\oplus\bb CD_-.$$ Moreover, the derivations $D_\pm$ induces a weight decomposition of the algebras $\al{L}(\ak{sl}_r,\sigma^\pm)\oplus \bb CD_\pm$. The main observation is that the above isomorphism does not respect the decompositions of these algebras in powers of $t$.\\ We will see in a moment that under the above isomorphism, the fundamental weight $\lambda_0^+$ of $\hat{\al L}(\ak{sl}_r,\sigma^+)\oplus\bb CD$ is sent to twice the fundamental weight $\lambda_0^-$.





\subsection*{Twisted conformal blocks} 
Let $\lambda_0^\pm,\cdots,\lambda_l^\pm$  be the fundamental wights of the twisted affine Lie algebras $\widehat{\al L}(\ak{sl}_r,\sigma^\pm)$, i.e. $\lambda_i^\pm$ are linear forms on the Cartan subalgebras such that $$\lambda^+_i(\beta_j)=\lambda_i^-(\alpha_j)=\delta_{ij},\; i,j=0,\dots,l.$$  Denote by $\mathring{\ak g}\subset \al{L}(\ak{sl}_r,\sigma^\pm)$ the simple Lie algebra generated by $e_i$ and $f_i$ for $i=1,\cdots,l$.  Note that $\mathring{\ak g}$ is of type $D_l$ in the case of $\sigma^+$ when $r$ is even, and it is of type $C_l$ otherwise. \\ 
Moreover, we have the identifications (see \cite[\S$12.4$]{K}) $$\lambda_i^\pm=\mathring{\lambda}_i+a_i^\vee\lambda_0^\pm, \;\ i=1,\dots,l\;,$$ where $\mathring{\lambda}_i$ ($i=1,\dots,l$) are the fundamental weights of $\mathring{\ak{g}}$.\\ 

\begin{rema} \label{level}
	Remark that, for an even rank $r$, the weight $\lambda_0^+$ has level equals $a_0^\vee=2$, while $\lambda_0^-$ has level $a_0^\vee=1$ (see Table \ref{table}). \\
\end{rema}

Denote by $\x{P}^{\sigma,\pm}$ the set of dominant integral weights of $\widehat{\al{L}}(\ak{sl}_r,\sigma^\pm)$. By \cite[\S$12.4$]{K},  one deduces a bijection between $\x{P}^{\sigma,\pm}$ and the set $$\tilde{\x{P}}^{\sigma,\pm}=\{(\lambda,k)|\lambda\in \mathring{\x{P}}, \;\gen{\lambda,\varrho}\leqslant k\},$$ where $\mathring{\x{P}}$ is the set of dominant weights of $\mathring{\ak{g}}$, and  $\varrho$ is the highest coroot of $\mathring{\ak g}$ when $r$ is even, and $\varrho$ is twice the highest coroot of $\mathring{\ak{g}}$ when $r$ is odd. \\

For $\mu^\pm\in \x{P}^{\sigma,\pm}$, denote by $\al H_{\mu^\pm}(k)$ the irreducible highest weight module of level $k$ of $\widehat{\al L}(\ak{sl}_r,\sigma^\pm)$ of highest weight $\mu^\pm$. Let $\overrightarrow{\mu}^\pm=(\mu_1^\pm,\cdots,\mu_{2n}^\pm)$ be a vector of elements of $\x{P}^{\sigma,\pm}$ parameterized by the points of $R$, and define $$\al H_{\overrightarrow{\mu}^\pm}(k)=\al H_{\mu_1^\pm}(k)\otimes\cdots\otimes \al H_{\mu_{2n}^\pm}(k).$$ 
Finally, let $\sr A_R=H^0(X\sm R,\al O_X)$. By considering the associated Lorrent series at $p\in R$, we get an inclusion $\ak{sl}_r(\sr A_R)^{\sigma^\pm}\subset\ak{sl}_r(\sr K_p)^{\sigma^\pm}$. We can than define an action of $\ak{sl}_r(\sr A_R)^{\sigma^\pm}$ on  $\al H_{\overrightarrow{\mu}^\pm}(k)$ as product of representations (i.e diagonal action).  More explicitly, for $\alpha\in\ak{sl}_r(\sr A_R)^{\sigma^\pm}$ and $X=X_1\otimes\cdots\otimes X_{2n}$, we have $$\alpha\cdot X=\sum_i X_1\otimes\cdots\otimes \alpha\cdot X_i\otimes\cdots\otimes X_{2n}.$$  

\begin{defi}
	The \emph{conformal block} attached to the data ($X$, $\sigma$, $\overrightarrow{\mu}^\pm$, $\widehat{\al{L}}(\ak{sl}_r,\sigma^\pm)$,$k$) is defined by $$\al V_{\sigma,\pm}(k)=\left[\left(\al H_{\overrightarrow{\mu}^\pm}(k)\right)_{\ak{sl}_r(\sr A_R)^{\sigma^\pm}}\right]^*,$$ 
	where for a $\ak{g}-$module $V$, we denote by  $V_{\ak{g}}$ the space of coinvariants of $V$, thus the largest quotient of $V$ on which $\ak g$ acts trivially.	
\end{defi}


\section{Loop groups and uniformization theorem}

\subsection{Bruhat-Tits parahoric $\al G-$torsors}
Let $\al G$ be a smooth affine group scheme over $X$. $\al G$ is said to be a parahoric Bruhat-Tits  group scheme  if there is a finite subset $R\subset X$ such that if $\al O_x$ is the completion of the local ring at $x\in R$ then  $\al G_{\al O_x}$ is a parahoric group scheme over $\x{Spec}(\al O_x)$ (in the sens of Bruhat-Tits, \cite[D\'efinition $5.2.6$]{BT2}) for each $x\in R$ and the  fibers $\al G_y$ is semisimple for all $y\in X\sm R$. \\ 
A class of examples of such group schemes is provided by the invariant Weil restriction. Given a Galois cover $\pi:X\ra X/\Gamma$ of curves and a semisimple algebraic group $G$ over $X$ with an action of $\Gamma$ lifted from its action on $X$, then $\al G=\pi_*^\Gamma (G)$ is a parahoric group scheme over $X/\Gamma$ (provided it is not empty).  Moreover, it is shown in \cite{BS} that the stack of $\Gamma-$equivariant $G-$torsors over $X$ is in one to one correspondence with the stack of $\al G-$torsors over $X/\Gamma$. \\ In our case, we have $\Gamma=\bb Z/2=\gen{\sigma}$ and $G=\x{SL}_r$.  Consider the actions of $\sigma$ on $\x{SL}_r$ given by  $$\sigma^+(g)=\,^tg^{-1},\;\; \sigma^-=J_r\sigma^+J_r^{-1},$$
where $$J_r=\begin{pmatrix}0 & I_{r/2} \\ -I_{r/2} & 0\end{pmatrix},$$  $I_{r/2}$ is the identity matrix of size $r/2$. 

 Let $\al G$ and $\al H$ the invariant Weil restrictions of the constant group scheme $\xb{SL}_r=X\times \x{SL}_r$ defined by $$\al G=(\pi_*\xb{SL}_r)^{\sigma^+}\,,\;\;\al H=(\pi_*\xb{SL}_r)^{\sigma^-}.$$
These two are smooth affine group schemes over $Y$ which are parahoric. Denote by $\sr{M}_Y(\al G)$ and $\sr{M}_Y(\al H)$ the stacks of $\al{G}-$torsors and  $\al H-$torsors over $Y$. By \cite[Proposition $2.4$]{Z2}, we have isomorphisms $$\sr{SU}_X^{\sigma,+}(r)\cong \sr M_Y(\al G)\,,\;\;\sr{SU}_X^{\sigma,-}(r)\cong \sr M_Y(\al H).$$

\subsection{Uniformization theorem}
For a ramification point  $p\in X$, denote by $\sr O_p$ the completion of the local ring at $p$,  $\sr K_p$  its fraction field and $\sr V_p$ a complementary vector subspace of $\sr O_p$ in $\sr K_p$. Let $\sr{SU}_X(r)$ denote the moduli stack of rank $r$ vector bundles over $X$ with a trivialization of its determinant. Let's fix the canonical linearization on $\al O_X$, so we identify $\sigma^*\al O_X$ and $\al O_X$. Moreover, since all the types are isomorphic, we assume hereafter that $\tau=(+1,\cdots,+1)\mod \pm1$ and denote the corresponding moduli stack by $\sr{SU}_X^{\sigma,-}(r)$.  \\ In \cite{LB}, it is proved that $$\sr{SU}_X(r)\cong\x{\textbf{SL}}_r(\sr O_p)\backslash\x{\textbf{SL}}_r(\sr K_p)/ \x{\textbf{SL}}_r(\sr A_p),$$ 
where $\sr A_p=H^0(X-p,\al O_X)$. Let  $t$ be  a local parameter at  $p$, then $\sr K_p\cong \bb C((t))$, $\sr O_p\cong \bb C [[t]]$. 

Consider the two involutions $\sigma^\pm$ on $\x{\textbf{SL}}_r(\sr K_p)$ given by  $$g(t)\ra \sigma^+(g(t))=\,^tg(-t)^{-1},$$ $$g(t)\ra\sigma^-(g(t))=J_r\cdot \,^tg(-t)^{-1} \cdot J_r^{-1}.$$ 

Let $\sr Q= \x{\textbf{SL}}_r(\sr O_p)\bs\x{\textbf{SL}}_r(\sr K_p) $. In \cite{PR2}, it is proved that $$\sr Q^{\sigma^+} =\xb{SL}_r(\sr O_p)^{\sigma^+}\bs\xb{SL}_r(\sr K_p)^{\sigma^+}.$$
Note that $\xb{SL}_r(\sr{O}_p)^{\sigma^+}$ is the maximal parahoric subgroup of $\xb{SL}_r(\sr{K}_p)^{\sigma^+}$ and, with the notations of loc. cit. this case corresponds to $I=\{0\}$. In fact, their involution is the conjugation of $\sigma^+$ by the anti-diagonal  matrix $D_r$  with all entries equal $1$. But this does not change much. Indeed, by taking a matrix $A$ such that $D_r=\,^tAA$ (such matrix can be constructed easily), then conjugation by $A$ realizes an isomorphism between $\xb{SL}_r(\sr K)^{\sigma^+}$ and their invariant locus. \\
We denote in the sequel by $\sr Q^{\sigma,\pm}$ the quotient  $\xb{SL}_r(\sr O_p)^{\sigma^\pm}\bs\xb{SL}_r(\sr K_p)^{\sigma^\pm}$, for some $p\in R$. 

\begin{theo}
	We have an isomorphism of stacks 
	$$\sr{SU}_X^{\sigma,\pm}(r)\cong \xb{SL}_r(\sr O_p)^{\sigma^\pm}\backslash\xb{SL}_r(\sr K_p)^{\sigma^\pm}/ \xb{SL}_r(\sr A_p)^{\sigma^\pm}.$$ 
	Moreover, the projections $\sr Q^{\sigma,\pm}\ra \sr{SU}_X^{\sigma,\pm}(r)$ 
	are locally trivial for the fppf topology.
\end{theo}
\begin{proof} Recall that we have isomorphisms $$\sr{SU}_X^{\sigma,+}(r)\cong \sr M_Y(\al G)\,,\;\;\sr{SU}_X^{\sigma,-}(r)\cong \sr M_Y(\al H).$$  Now, using the main Theorem of \cite{He}, we deduce, for a ramification point $p\in X$ over a branch point $y\in Y$, that 
	\begin{align*}
		\sr M_Y(\al G)&\cong \al G(\sr O_y)\bs\al G(\sr K_y)/H^0(Y\sm y,\al G)\\
		&\cong \xb{SL}_r(\sr O_p)^{\sigma^+}\bs\xb{SL}_r(\sr K_p)^{\sigma^+}/\xb{SL}_r(\sr A_p)^{\sigma^+}.
	\end{align*}	
	$$\sr M_Y(\al H)\cong \al{H}(\sr O_y)\bs\al{H}(\sr K_y)/H^0(Y\sm y,\al H).$$
	And we have $H^0(Y\sm y,\al H)\cong \xb{SL}_r(\sr A_p)^{\sigma^-}$. Thus 	$$\sr M_Y(\al H)\cong \xb{SL}_r(\sr O_p)^{\sigma^-}\bs\xb{SL}_r(\sr K_p)^{\sigma^-}/\xb{SL}_r(\sr A_p)^{\sigma^-}.$$ 
\end{proof}

\subsection{The Grassmannian viewpoint}\label{grass}
Note that $\sr Q^{\sigma,+}$ is an ind-variety, which is a direct limit of a system of projective varieties $(\sr Q^{\sigma,+}_N)_{N\geqslant 0}$, the $\sr Q^{\sigma,+}_N$ are the quotients $(S^0)^{\sigma^+}\bs (S^N)^{\sigma^+}$, where $S^N$ is the subscheme of $\xb{SL}_r(\sr K)$ parameterizing matrices $A(t)$ such that $A(t)$ and $A(t)^{-1}$ have poles of order at most $N$. As we said above, since all the stacks $\sr{SU}_X^{\sigma,\tau}(r)$ are isomorphic, so for simplicity we assume that $\tau$ is the trivial type.  So let's denote  $\sr Q^{\sigma,-}$ the quotient $\xb{SL}_r(\sr O_p)^{\sigma^{-}}\bs\xb{SL}_r(\sr K_p)^{\sigma^{-}}$. This is again an ind-variety direct limit of $(\sr Q^{\sigma,-}_N)_{N\geqslant 0}$, the $\sr Q^{\sigma,-}_N$ are the quotients $(S^{0})^{\sigma^{-}}\bs(S^{N})^{\sigma^{-}}$. 

By \cite[Proposition $2.4$]{LB} , the varieties $\sr Q_N:=S^0\bs S^{N}$ are  identified with  subvarieties (with the same underline topological spaces)  of the Grassmannian $\x{Gr}^{t}(rN,2rN)$ of $t-$stable subspaces of dimension $rN$ of $F_N^r:=t^{-N}\sr O^{\oplus r}/t^{N}\sr O^{\oplus r}$. 

Consider   the $\sigma-$Hermitian forms $\Psi_\pm:\sr K^{r}\times \sr K^{r}\lra \sr K$  defined by $$\Psi_+(v,w)=\,^tv\cdot\sigma(w)= \sum_{i=1}^{r}v_i\sigma(w_i),$$ $$\Psi_-(v,w)= \,^tv\cdot J_r\cdot \sigma(w),$$ where $v=\,^t(v_1,\cdots,v_i)$ and $w=\,^t(w_1,\cdots,w_i)$ are in $\sr K^r$. Then the groups   $\x{SL}_r(\sr K)^{\sigma^\pm}$ can be defined as the loci of matrices $A\in\x{SL}_r(\sr K)$ which are unitary with respect to  the forms  $\Psi_\pm$, i.e. $\Psi_{\pm}(A\cdot v,A\cdot w)=\Psi_{\pm}(v,w)$ for all $v,w\in\sr K^r$.  \\  Consider the forms  $\tilde{\Psi}_\pm$ on $t^{-N}\sr O^r\subset \sr K^r$ defined as the composition $$\tilde{\Psi}_\pm: t^{-N}\sr O^{\oplus r}\times t^{-N}\sr O^{\oplus r}\xrightarrow{\Psi_\pm} t^{-2N}\sr O\xrightarrow{\x{Res}}\bb C,$$ where $\x{Res}:\sr K\ra \bb C$ is the residue map.  
The forms $\tilde{\Psi}_\pm$ vanish on $t^N\sr O^{\oplus r}\subset t^{-N}\sr O^{\oplus r}$,  hence they induce  two forms, denoted again by $\tilde{\Psi}_\pm$, on $F_N^r$
$$\tilde{\Psi}_\pm:F_N^r\times F_N^r\lra \bb C.$$ 
\begin{lemm}  $\tilde{\Psi}_+$ is an anti-symmetric non-degenerate bilinear form  on $F_N^r$, while  $\tilde{\Psi}_-$ is a symmetric non-degenerate bilinear form. 
\end{lemm}
\begin{proof}
	Consider the canonical basis of the vector space  $F_r$ given by the classes of $t^k$ for $k=-N,\cdots,N-1.$ It induces a canonical basis of $F_N^r$. Then for $v=(v_i)_i, w=(w_i)_i\in F_N^r$, the forms $\Psi_\pm$ are given explicitly in this basis by $$\Psi_+(v,w)=\sum_{i=1}^{r}\sum_{j=-N}^{N-1}(-1)^{-j-1}a_j^ib_{-j-1}^i,$$ $$\Psi_-(v,w)=\sum_{i=1}^{r}\sum_{j=-N}^{N-1}(-1)^{-j-1+\varepsilon(i)}a_j^ib_{-j-1}^{r-i},$$ 
	where $\varepsilon(i)$ equals $1$ if $i\in\{1,\cdots,r/2\}$,  and $0$ otherwise, and $v_i=(a_j^i)$, $w_i=(b_j^i)$ are in $F_N$.  From this the result follows easily. 
\end{proof}
 
\begin{prop} \label{isotropic}
	The spaces $\sr Q^{\sigma,\pm}_N$ are isomorphic to  closed subvarieties (with the same underline topological subspaces) of the isotropic Grassmannian $\x{Gr}^{t,\sigma}_\pm(rN,2rN)$ which parameterizes $\tilde{\Psi}_\pm-$isotropic $t-$stable vector subspaces  of $F_N^r$ of dimension $rN$. 
\end{prop}
\begin{proof}
	We prove it for the symmetric case, the other one follows  similarly.\\
	The image of $\sr O^{\oplus r}$ in $F_N^r$ is $\tilde{\Psi}_+-$isotropic, hence, for every $A(t)\in (S^{N})^{\sigma^+},$ the corresponding point in $\x{Gr}^{t}(rN,2rN)$ of the class of $A(t)$ in $\sr Q^{\sigma,+}_N$ is $\tilde{\Psi}_+-$isotropic. Thus it is in $\x{Gr}^{t,\sigma}_+(rN,2rN)$. 
	
	Conversely, assume that we have a point $W$ of the isotropic Grassmannian. Let $A(t)\in S^N$ be a representative of the corresponding class in $S^0\bs S^N$. We have for every $v,w\in\sr O^r$,  $\Psi_+(A\cdot v,A\cdot w)\in \sr O$, to see this assume that for some $v,w\in \sr O^r$, the coefficient of $t^{-k}$ of  $\Psi_+(A\cdot v,A\cdot w)$ is nonzero ($k>0$), then one deduces that $\x{Res}(\Psi_+(A\cdot (t^kv),A\cdot w))=\tilde{\Psi}_+(A\cdot (t^kv), A\cdot w)\not=0$, contradiction. Now let $(e_i)_i$ be the canonical basis of the $\sr K-$vector space $\sr K^r$ and let $B(t)=(\Psi_+(A\cdot e_i,A\cdot e_j))_{i,j}$, we see that $B(t)\in \x{SL}_r(\sr O)=S^0$, and we have by definition  $B(t)=\,^tA(t)A(-t)$. In particular we see that  $B(t)=\,^tB(-t)$, hence $B(t)=\,^tC(-t)\cdot C(t)$ for some $C(t)\in \x{SL}_r(\sr O)$, and $C(t)A(t)$ is also a representative of $W$ and it is of course in $(S^{N})^{\sigma^+}$. In other words the corresponding point of $W$  in $S^0\bs S^N$ is in $\sr Q^{\sigma,+}_N$. This proves the proposition.
\end{proof}

Consider the variety $\sr Q_N=S^0\bs S^N$ which is as a topological space  isomorphic to  the Grassmannian $\x{Gr}^t(rN,2rN)$. Fix an identification of $\sr Q_N$ as subspace of the homogeneous space $\x{SL}_{2rN}(\bb C)/\x{P}_N$, where $\x{P}_N$ is the parabolic subgroup of $\x{SL}_{2rN}$  of matrices of  the form $$\begin{pmatrix}A&B \\ 0& C \end{pmatrix},$$ where $A,B$ and $C$ are square $rN\times rN$ matrices. Let $\al O_{\sr Q_N}(1)$ be the line bundle attached to the character $\chi:\x{P}\ra \bb C^*$ which sends a matrix as above to $\x{det}(A^{-1})$. It is well known that the Picard group of $\sr Q_N$ is infinite cyclic generated by $\al O_{\sr Q_N}(1)$ (it is actually isomorphic to the character group of the maximal parabolic subgroup $P_N$).
\begin{prop} \label{squareroot}
	The restriction of $\al O_{\sr Q_N}(1)$ to $\sr Q^{\sigma,-}_N$ has a square root, which we denote by $\al O_{\sr Q^{\sigma,-}_N}(1)$. It is in fact the generator of the Picard group of $\sr Q^{\sigma,-}_N$.
\end{prop}
\begin{proof}
	By Proposition \ref{isotropic}, the variety  $\sr Q^{\sigma,-}_N$ is isomorphic to a subvariety of the classical Grassmannian  $\x{SO}_{2rN}(\bb C)/\x{P}'_N$, where $\x{P}'_N=\x{P}_N\cap\x{SO}_{2rN}(\bb C)$. The restriction of the character $\chi$ to $\x{P}'_N$ is denoted again by $\chi$. Now, consider the universal cover of $\x{SO}_{2rn}(\bb C)$, which is  the Spin group $\x{Spin}_{2rN}(\bb C)$. It is a  double  cover of  $\x{SO}_{2rN}(\bb C)$. Let $\tilde{\x{P}}_N\subset\x{Spin}_{2rN}(\bb C)$ the inverse image of $\x{P}'_N$. Then, by \cite[Chapter $5$, Theorem $3.3.1$]{DSS}, the lifting of $\chi$ to $\tilde{\x{P}}_N$ has a square root which we denote  by $\chi_{-}$.  Since we have $$\x{Spin}_{2rn}(\bb C)/\tilde{\x{P}}\cong \x{SO}_{2rn}/\x{P}'_N,$$ we deduce that  the line bundle over $\sr Q^{\sigma,-}_N$ attached to $\chi_{-}$ is the square root  of the restriction of $\al O_{\sr Q_N}(1)$. \\ 
	The Picard group of $\x{Spin}_{2rN}/\tilde{\x{P}}_N$ is infinite cyclic isomorphic to the character group of $\tilde{\x{P}}_N$, which is generated by $\chi_{-}$. This implies the second claim.
\end{proof}

\begin{prop}\label{reduced}
	The ind-varieties $\sr{Q}^{\sigma,\pm}$ are integral.
\end{prop}
\begin{proof} We know already that $\sr Q^{\sigma,\pm}$ are connected, hence they are irreducible. Moreover, by \cite[Theorem $0.2$]{PR2}, the flag varieties $\sr Q^{\sigma,\pm}$ are reduced. 
\end{proof}

\subsection{Central extension}

Consider the \emph{canonical} central extension of $\xb{SL}_r(\sr K)$ defined in [\cite{BL}, Section $4$]: $$0\ra \bb G_m\ra\widehat{\xb{SL}_r}(\sr K)\ra \xb{SL}_r(\sr K)\ra 0.$$
The actions of  $\sigma^\pm$ lift to $\widehat{\xb{SL}_r}(\sr K)$ giving a central extension of $\xb{SL}_r(\sr K)^{\sigma^\pm}$  
$$ 0\ra \bb G_m\ra\widehat{\xb{SL}_r}(\sr K)^{\sigma^\pm}\ra \xb{SL}_r(\sr K)^{\sigma^\pm}\ra 0.$$ 
Indeed, let $R$ be a $\bb C-$algebra, for $\gamma\in\x{SL}_r(R((t)))$ let 
$$\gamma=\begin{pmatrix} a(\gamma) & b(\gamma)\\ c(\gamma) & d(\gamma) \end{pmatrix}$$ 
be its decomposition with respect to $R((t))=\sr V_R\oplus R[[t]]$. Recall that $\sr V$ is a complementary vector subspace of $\sr O$ in $\sr K$.\\
By \cite{LB}, an element of $\x{SL}_r(R((t)))$ is given, locally on $\x{Spec}(R)$, by a pair $(\gamma,u)$ where $\gamma\in \x{SL}_r(R((t)))$, $u\in \x{Aut}(\sr V_R)$ such that $u\equiv a(\gamma) \mod \x{End}^f(\sr V_R)$, where $\x{End}^f(\sr V_R)\subset \x{End}(\sr V_R)$ is the set of finite rank endomorphisms of $\sr V_R$. By \cite{LB} Proposition $4.3$,  the map $\gamma\lra \overline{a}(\gamma)$ is a group homomorphism from $\x{SL}_r(R((t)))$ onto the group $\x{Aut}^f(\sr V_R)$ of  units of $\x{End}(\sr V_R)/\x{End}^f(\sr V_R)$. It follows that $$\overline{a}(\gamma^{-1})=\overline{a}(\gamma)^{-1},$$
hence $$u^{-1}\equiv a(\gamma^{-1})\mod \x{End}^f(\sr V_R).$$
So, define the following actions on $\widehat{\xb{SL}_r}(\sr K)$ $$\sigma^+:(\gamma,u)\lra(\,^t\gamma(-t)^{-1},\,^tu(-t)^{-1}),$$ $$\sigma^-:(\gamma,u)\lra(\,J_r\,^t\gamma(-t)^{-1}J_r^{-1},J_r\,^tu(-t)^{-1}J_r^{-1}).$$
Clearly these  are  involutions which lift $\sigma^\pm$ on $\xb{SL}_r(\sr K)$. \\

The Lie algebra attached to $\widehat{\xb{SL}}_r(\sr K)$ is given by  the central extension
\begin{equation} \label{eq1}0\ra\bb C\ra\widehat{\ak{sl}_r}(\sr K)\ra \ak{sl}_r(\sr K)\ra 0.\end{equation} It is in fact  isomorphic to the affine Lie algebra  $\hat{\al L}(\ak{sl}_r)=\ak{sl}_r(\sr K)\oplus \bb C$, with the Lie algebra structure given by  $$[(\alpha,u),(\beta,v)]=\left([\alpha;\beta],\x{Res}(\x{Tr}(\dfrac{d\alpha}{dt}\beta))\right),$$ 
where $\x{Res}$ stands for the \emph{residue}. 
By pulling back the exact sequence (\ref{eq1}) via the inclusions $\ak{sl}_r(\sr K)^{\sigma^\pm}\hookrightarrow \ak{sl}_r(\sr K)$ we get the central extensions $$ \xymatrix{0\ar[r]&\bb C\ar[r]&\widehat{\ak{sl}_r}(\sr K)\ar[r]& \ak{sl}_r(\sr K) \ar[r]&  0\\0\ar[r]&\bb C\ar[r]\ar@{=}[u]&\widehat{\ak{sl}_r}(\sr K)^{\sigma^\pm}\ar[r]\ar[u]& \ak{sl}_r(\sr K)^{\sigma^\pm}\ar[u] \ar[r]&  0,}$$ where $\sigma^\pm$ act on $\widehat{\ak sl_r}(\sr K)$ by their  actions on the first summand (which are given in Lemma \ref{actions} below). These are (after adding scaling elements) affine Kac-Moody Lie algebras of twisted type $A_{r-1}^{(2)}$. They are in fact the Lie algebras of the twisted  groups $\widehat{\xb{SL}}_r(\sr{K})^{\sigma^\pm}$.
\begin{lemm}\label{actions}
	The Lie algebras associated to $\widehat{\xb{SL}}_r(\sr K)^{\sigma^\pm}$ are the twisted affine Lie algebras of $\ak{sl}_r(\sr K)$ given by $$\widehat{\al{L}}(\ak{sl}_r,\sigma^\pm)=\ak{sl}_r(\sr K)^{\sigma^\pm}\oplus\bb C,$$ where the actions of $\sigma^\pm$ on $\ak{sl}_r(\sr K)$ are given by $$\sigma^+(g(t))= -\,^tg(-t),$$
	$$\sigma^-(g(t))= -J_r\,^tg(-t)J_r^{-1}.$$
\end{lemm}
\begin{proof}
	The proof is straightforward, 
	we just remark  that $$^t(I_r+\varepsilon \alpha)^{-1}=I_r-\varepsilon \;^t\alpha,$$ where $\varepsilon^2=0$. 
\end{proof}
\section{Determinant and Pfaffian line bundles}
Let $T$ be a locally noetherian $\bb C-$scheme. Denote by $p_1$ and $p_2$ the projection maps from $X\times T$ to $X$ and $T$ respectively. Let  $\sr E$ be a vector bundle over $X\times T$. The derived direct image complex ${Rp_2}_*(\sr E)$ is represented by a complex of vector bundles $0\ra F_0\ra F_1\ra 0$. The line bundle $\al D_{\sr E}:=\x{det}(F_0)^{-1}\otimes \x{det}(F_1)$ over $T$ is independent of the choice of the representing complex and is called the \emph{determinant of cohomology} of $\sr E$. The determinant of  the universal family $\sr L$ over $X\times \sr{SU}_X(r)$ is called the \emph{determinant bundle} over  $\sr{SU}_X(r)$.

\subsection{Ramified case} Assume that $\pi:X\ra Y$ is ramified along a divisor $R$ and let  $\x{P}^-=\x{Nm}^{-1}(K_Y\Delta)$. 
\begin{prop}\label{pfaffian}
	Let $(\sr E,\psi)$ be a  family  of $\sigma-$alternating  vector bundles over $X$ parameterized by $T$,  with a $\sigma-$alternating non-degenerated form $\psi: \sigma^*\sr E\lra \sr E^*$. For any $L\in\x{P}^-$ let $\sr E_L=\sr E\otimes p_1^*L$.  Then the determinant of cohomology line bundle $\al D_{\sr E_L}$ admits a square root $\al P_{\sr E_L}$ which we call \emph{Pfaffian of cohomology} line  bundle. 
\end{prop}
\begin{proof}
	Consider the family $\pi_*\sr E_L$ over $Y$. It is equipped with a non-degenerated quadratic form with values in $K_Y$. Indeed, by the projection formula, $\psi$ induces an isomorphism 
	\begin{align*}
		\pi_*\sr E_L & \cong \pi_*(\sigma^*\sr E_L) \\& \cong \pi_*(\sr E^*_L (q_1^{-1}(R)))\otimes q_1^* K_Y \\ &\cong (\pi_*\sr E_L)^*\otimes q_1^* K_Y,\end{align*}
	where the last isomorphism is the relative duality (see \cite[Ex. III.6.10]{HA}) and $q_1:Y\times T\ra Y$ is the first projection.  In fact the associated bilinear form is given by the composition $$\pi_*\sr E_L\otimes \pi_*\sr E_L\lra \pi_*(p_1^*K_X)=q_1^*(K_Y\otimes\Delta)\oplus q_1^*K_Y\lra q_1^*K_Y.$$ Since we project on the $-1$ eigenspace of the linearization on $K_X$ (recall that $\pi_*K_X=K_Y\Delta\oplus K_Y$)  and because $\psi$ is $\sigma-$alternating, we deduce that this bilinear form is symmetric.
	
	We can apply now \cite[Proposition $7.9$]{LS}  to get a square root of $\al D_{\pi_*\sr E_L}$. To finish the proof we just have to remark that $$\al D_{\sr E_L}=\al D_{\pi_*\sr E_L}.$$
\end{proof}

In particular, if we consider the universal family over $X\times \sr{SU}_X^{\sigma,-}(r)$, we get, for each $L\in \x{P}^-$, a Pfaffian of cohomology line bundle $\al P_L$ over $\sr{SU}_X^{\sigma,-}(r)$. 

On the other hand, consider the character $\chi:\widehat{\xb{SL}_r}(\sr O)\ra \bb G_m$ which is just the second projection (recall that $\widehat{\xb{SL}_r}(\sr O)$ splits). More precisely,  a point of $\widehat{\xb{SL}_r}(\sr O)$ can be represented locally on $\x{Spec}(R)$ by a pair $(\gamma,u)$, for $\gamma\in\x{SL}_r(R[[t]])$ and $u$ an automorphism of $\sr V_R$ such that $a(\gamma)\equiv u\mod \x{End}^f(\sr V_R)$. So $\chi$ sends this point to $\x{det}(a(\gamma)^{-1}u)$.   To this character one may associate a line bundle $\al L_{\chi}$ over $\sr Q$ (see \cite[\S3]{LB}). Moreover $\al{L}_\chi$ is isomorphic to the pullback of the determinant bundle.

\begin{lemm} 
	The restriction of the character $\chi$ to $\widehat{\xb{SL}_r}(\sr O)^{\sigma^-}$ has a square root which we denote by $\chi_{-}$. 
\end{lemm}
\begin{proof}
	With the notations of the proof of Proposition \ref{squareroot}, one can see that $\xb{SL}_r(\sr O)^{\sigma^-}$ is the direct limit of the parabolic subgroups $\x{P}'_N$. So just take the direct limit in  
	Proposition \ref{squareroot}. 
\end{proof}
Let $\al L_{-}$ be the line bundle  over $\sr Q^{\sigma,-}$ defined by the character $\chi_{-}$ and  denote by  $q:\sr Q^{\sigma,-}\lra \sr{SU}_X^{\sigma,-}(r)$ the quotient maps (there should be no confusion about which map is considered).\\ Denote by $\al D$ the determinant of cohomology line bundle over $\sr{SU}_X^{\sigma,-}(r)$ and by $\al L$ its pullback to $\sr Q^{\sigma,-}$.




\begin{prop} \label{picardgroup}
Let $r\geqslant 3$. Then the isomorphism class of the Pfaffian bundles $\al P_L$  is independent  of $L$, so we just denote it by $\al P$, and its  pullback $q^*\al P$ is isomorphic to $\al L_{\chi_-}$. Moreover we have $$\x{Pic}(\sr{SU}_X^{\sigma,+}(r))=\bb Z\al D,\;\;\x{Pic}(\sr{SU}_{X,0}^{\sigma,-}(r))=\bb Z\al P.$$
\end{prop}
\begin{proof}
  Since $\x{P}^-$ is connected and the group $Pic(\sr{SU}_X^{\sigma,-}(r))$ is discrete, we deduce that the map $L\ra \al P_L$ is constant. Hence all $\al P_L$ are isomorphic. \\ 
  Now, by  \cite[Proposition $2.3$]{Z2}, we have $$0\ra \x{Sym}_r^0(\bb C)\stackrel{i}{\ra} \al G_p\ra \x{SO}_r(\bb C)\ra 0.$$ Even though that this extension does not split, $\x{SO}_r$ can be embedded in $\al G_p$ via the map $g\ra g+\varepsilon \times 0$. Let  $X^*(\al G_p)$ is the character group of $\al G_p$ and $\lambda\in X^*(\al G_p)$. Since $\x{SO}_r(\bb C)$ is semi-simple, the restriction of $\lambda$ to $\x{SO}_r(\bb C)$ is trivial. Moreover, the restriction of $\chi$ to $i(\x{Sym}_r^0(\bb C))$ is trivial too because $\x{Sym}_r^0(\bb C)$ is a unipotent.

	Now, by \cite[Theorem $3$]{He}, we have $$0\ra \Pi_{p\in B}X^*(\al G_p)\ra \x{Pic}(\sr M_Y(\al G))\ra \bb Z\ra 0,$$ it follows that  $\x{Pic}(\sr{SU}_X^{\sigma,+}(r))\cong \bb Z$. It is well known that the pullback of the determinant line bundle over $\sr{SU}_X(r)$ to $ \sr Q$ is  $\al L_\chi$ (see for example \cite{LS}). Furthermore, $\al D$ has no square root. Indeed, by \cite[Theorem $4.16$]{Z}, the stack $\sr{SU}_X^{\sigma,+}(r)$ is dominated by the intersection of two Prym varieties $\tilde{ \x{P}}\cap \tilde{\x{Q}}$ which is not principally polarized. Hence if $\al D$ has a square root, it follows that the polarization of $\tilde{\x{P}}\cap\tilde{\x{Q}}$ has also a square root, which is not true. Thus  $\al D$ is a generator of $\x{Pic}(\sr{SU}_X^{\sigma,+}(r))$.  

With the same method we deduce that $\al P$ has no root, this implies that  that  $\x{Pic}(\sr{SU}_X^{\sigma,-}(r))\cong \bb Z\al P$. 
\end{proof}
\begin{rema}
	The case $r=2$ is special. In this case, the moduli stack $\sr{SU}_X^{\sigma,+}(2)$ is connected and $\sr{SU}_X^{\sigma,-}(2)$ has $2^{2n-1}$ connected component. It is pointed out in \cite{Z} that these moduli stacks can be identified with some moduli of parabolic rank $2$  bundles on $Y$. Hence one can deduce their Picard groups using \cite{LS}.\\
\end{rema}

\subsection{Unramified case}\label{unramifiedPfaffian}Assume here that $\pi:X\ra Y$ is \'etale. Let $\x{P}^+=\x P_+^{ev}\cup\x P_+^{od}=\x{Nm}^{-1}(K_Y)$ and $\x{P}^-=P^{a}_-\cup P_-^{b}=\x{Nm}^{-1}(K_Y\Delta)$.  Note the subscript in the $+$ case is canonical and it corresponds to the parity of $h^0$ of the line bundles. Then we have the following 
\begin{prop}\label{pfaffianetale}
	Let $(\sr E^\pm,\psi)$ be a  family  of $\sigma-$symmetric (resp. $\sigma-$alternating) vector bundles over $X$ parameterized by $T$,  with a $\sigma-$symmetric (resp. $\sigma-$alternating) non-degenerated form $\psi: \sigma^*\sr E^\pm\lra {\sr E^\pm}^*$. For any $L\in \x{P}^+$ (resp. $L\in \x{P}^-$) let $\sr E^\pm_L=\sr E^\pm\otimes p_1^*L$.  Then the determinant of cohomology line bundle $\al D_{\sr E^\pm_L}$ admits a square root $\al P^\pm_{\sr E^\pm_L}$ which we call \emph{Pfaffian of cohomology} line  bundle. 
\end{prop}
\begin{proof} The proof is similar to that of Proposition \ref{pfaffian}. We note just that when $L\in\x{P}^+$, the norm map induces a quadratic form on $\pi_*L$, and when it is in $\x{P}^-$, the induced form  is alternating.
\end{proof}

Let $\sr U^\pm$ be  universal families on $\sr{SU}_X^{\sigma,+}(r)$ and $\sr{SU}_X^{\sigma,-}(2r)$, then for any $L\in P^\pm$, we have a Pfaffian of cohomology line bundle $\al P_L^\pm:=\x{Pf}(\sr{U}^\pm_L)$. Moreover, since $\x{Pic}(\sr{SU}_X^{\sigma,\pm}(r))$ are discrete groups (this can be deduced from the uniformization theorem for example), we deduce that there is at most two  isomorphism classes of $\al P^\pm_L$ parametrized by the connected components of $\x{P}^\pm$. So we denote them by $\al P_{ev}^\pm$ and $\al P_{od}^\pm$. \\

\section{Generalized theta functions and conformal blocks}

Assume in this section that the cover $\pi:X\ra Y$ is ramified. We have formulated the uniformization theorem over a single ramification point. However we can use a bunch of points to uniformize our moduli stack. If we consider all the ramification points $R$, then we get  the following   $$\sr{SU}_X^{\sigma,\pm}(r)\cong {\prod_{p\in R}\sr{Q}^{\sigma,\pm}_p}/\xb{SL}_r(\sr A_R)^{\sigma^\pm},$$ where $\sr{Q}^{\sigma,\pm}_p=\xb{SL}_r(\sr O_p)^{\sigma^\pm}\bs\xb{SL}_r(\sr K_p)^{\sigma^\pm}$, and $\sr{A}_R=H^0(X\sm R,\al O_X)$. Of course all $\sr{Q}_p^{\sigma,\pm}$ are isomorphic, but we emphasis on the fixed points. \\
Roughly speaking, this isomorphism  can be seen as follows: choose a formal neighborhood $D_p$ of each $p\in R$. Then giving a $\sigma-$symmetric vector bundle $(E,\psi)$ of trivial determinant, we choose a $\sigma-$invariant local trivializations $\varphi_p$ near each $p$ and a $\sigma-$invariant  trivialization $\varphi_0$ on $X\sm R$.  Then the corresponding point of the right hand side is just the class of $(\varphi_p\circ \varphi_0^{-1})_{p\in R}$. Conversely, giving a class of functions $(f_p)_{p\in R}$ of the RHS, we can construct a $\sigma-$symmetric vector bundle by gluing the trivial bundles on $D_p$ and $X\sm R$ using the functions $f_p$.\\

We have seen that the line bundle $\al L_{-}$ over $\sr{Q}_p^{\sigma,-}$ is isomorphic to $ q^*\al P$ and that the line bundle $\al L$ over $\sr{Q}_p^{\sigma,+}$ is isomorphic to $q^*\al D$. For $x\in R$, let $q_x:\prod_{p\in R}\sr{Q}^{\sigma,\pm}_p\ra \sr{Q}_x^{\sigma,\pm}$ be the canonical projection. We define the line bundles $$\sr L_-=\bigotimes_{p\in R}q_p^*\al{L}_- \x{  and  }  \sr L=\bigotimes_{p\in R}q_p^*\al L$$ over $\prod_{p\in R}\sr{Q}_p^{\sigma,-}$ and $\prod_{p\in R}\sr{Q}_p^{\sigma,+}$ respectively. One can see that $\sr L_-$ and $\sr L$ are in fact the pullback  via the projections $\prod_{p\in R}\sr Q_p^{\sigma,\pm}\ra \sr{SU}_X^{\sigma,\pm}(r)$ of the line bundles  $\al P$ and $\al D$ respectively.  In particular, both of these line bundles  have canonical $\xb{SL}(\sr A_R)^{\sigma^\pm}-$linearizations. In fact these are the only ones due to the following
\begin{prop}
	$\xb{SL}_r(\sr A_R)^{\sigma^\pm}$ are integral and they have only the trivial character.
\end{prop}
\begin{proof}The proof is inspired from \cite{LS}. \\ 
	Using the local triviality of the projection $\prod_{p\in R}\sr{Q}_p^{\sigma,\pm}\ra\sr{SU}^{\sigma,\pm}_X(r)$ and Proposition \ref{reduced} we deduce that $\xb{SL}_r(\sr A_R)^{\sigma^\pm}$ are reduced. 
	
	Now, since connected ind-groups are irreducible, it is sufficient to prove that $\xb{SL}_r(\sr A_R)^{\sigma^\pm}$ is connected. For a points $p_1, \dots,p_k\in X\sm R$ we denote by $R_i=R\cup\{p_1,\sigma(p_1),\dots,p_i,\sigma(p_i)\}$. We claim the following 
	\begin{claim}
		We have an isomorphism $$\xb{SL}_r(\sr A_{R_{i}})^{\sigma^\pm}/\xb{SL}_r(\sr{A}_{R_{i-1}})^{\sigma^\pm}\cong (\sr Q_{p_{i}}\times \sr Q_{\sigma(p_{i})})^{\sigma^\pm},$$
		where the action of $\sigma^\pm$ on the right hand side  is given by $\sigma^\pm(f,g)=(\sigma^\pm(g),\sigma^\pm(f))$.
	\end{claim}
	\begin{proof}
		We have a canonical map $\xb{SL}_r(\sr A_{R_{i}})^{\sigma^\pm}\ra (\sr{Q}_{p_i}\times\sr Q_{\sigma(p_i)})^{\sigma^\pm}$ which is clearly trivial on $\xb{SL}_r(\sr A_{R_{i-1}})^{\sigma^\pm}$. Hence we deduce a map $ \xb{SL}_r(\sr A_{R_i})^{\sigma^\pm}/\xb{SL}_r(\sr A_{R_{i-1}})^{\sigma^\pm}\ra(\sr{Q}_{p_i}\times\sr Q_{\sigma(p_i)})^{\sigma^\pm} $ which is actually injective. Now, by considering the uniformization over the two points $\{p_i,\sigma(p_i)\}$, we get $$\sr{SU}_X^{\sigma,\pm}(r)\cong (\sr{Q}_{p_i}\times\sr Q_{\sigma(p_i)})^{\sigma^\pm}/\xb{SL}_r(\sr A_{\{p_i,\sigma(p_i)\}})^{\sigma^\pm}.$$
		Hence, for an $\bb C-$algebra $S$, giving a point of $(\sr{Q}_{p_i}\times\sr Q_{\sigma(p_i)})^{\sigma^\pm}(S)$ is the same as giving an anti-invariant ($\sigma-$symmetric or $\sigma-$alternating following $\pm$) vector bundle $E$ over $X_S$ and  a $\sigma^\pm-$invariant trivialization $\delta:E|_{X_S^*}\ra X_S^*\times \bb C^r$, where  $X_S^*=X_S\sm\{p_i,\sigma(p_i)\}$. For an $S-$algebra $S'$,  let $T(S')$ be the space of $\sigma^\pm-$invariant trivializations of $E_{S'}$ over $X_{S,i-1}=X_S\sm R_{i-1}$. Then $\xb{SL}_r(\sr A_{R_{i-1}})^{\sigma^\pm}$ acts on $T$, and in fact it is a torsor under that group. Moreover $\delta$ induces a map $\tilde{\delta}: T\ra \xb{SL}_r(\sr A_{R_i})^{\sigma^\pm}$ by sending a trivialization $\phi$ to $\phi\circ \delta^{-1}$.  Associating to $(E,\delta)$ the map $\tilde{\delta}$ gives an inverse to the above inclusion. 
	\end{proof}
	
	It is clear to see that $(\sr{Q}_{p_i}\times\sr Q_{\sigma(p_i)})^{\sigma^\pm}\cong \sr Q_{p_i}=\xb{SL}_r(\sr O_{p_i})\bs \xb{SL}_(\sr K_{p_i})$ which is simply connected. So using the homotopy  exact sequence, we deduce that $$\pi_0(\xb{SL}_r(\sr A_{R_{i}}))=\pi_0(\xb{SL}_r(\sr A_{R_{i-1}})).$$ 
	Now let $g\in \x{SL}_r(\sr A_R)^{\sigma^\pm}$ and consider $g$ as an element of $\x{SL}_r(\al K)^{\sigma^\pm}$, where $\al K$ is the function field of $X$. By \cite{T} (see also \cite{PR} Section $4$), we know that the special unitary groups are simply connected and quasi-split. Steinberg (\cite{St}) has showed the Kneser-Tits property for quasi-split simply connected groups over any field (Recall that this property means that these groups are generated by the unipotent radicals of their standard parabolic subgroups).  So applying that to  $\x{SL}_r(\al K)^{\sigma^\pm}$, we can assume that $g=\prod_{i}\exp(N_i)$, where 
	$N_i$ are nilpotent elements of $\ak{sl}_r(\al K)^{\sigma^\pm}$. Let $\{p_1,\dots,p_k\}$ be the set of poles of  the $N_i$'s. For $t\in \bb A^1$, we let  $g_t=\prod_{i}\exp(tN_i)$. Then for any $t\in \bb A^1$ we see that  $g_t\in \x{SL}_r(\sr A_{R_k})^{\sigma^\pm}$ and $t\ra g_t$ is a path in $\x{SL}_r(\sr A_{R_k})^{\sigma^\pm}$ that relates $g$ to the identity. Hence $\xb{SL}_r(\sr A_{R_k})^{\sigma^\pm}$ is connected. So the same is true for $\xb{SL}_r(\sr A_R)^{\sigma^\pm}$ by what we have shown above.\\
	
	
	Now let $\lambda$ be a character of $\x{SL}_r(\sr A_R)^{\sigma^\pm}$, seeing $\lambda$ as a function, we consider its derivative at the identity which turns out to be a Lie algebras morphism from $\ak{sl}_r(\sr A_R)^{\sigma^\pm}$ to the trivial algebra $\bb C$. However, the affine algebra  $\ak{sl}_r(\sr A_R)^{\sigma^\pm}$ equals the direct sum of two commutator subalgebras. Indeed, the algebra  $\ak{sl}_r(\sr A_R)$ equals to its commutator, and we have eigenspace decomposition with respect to $\sigma^\pm$ $$\ak{sl}_r(\sr A_R)=\ak g_{-1}\oplus \ak{g}_1,$$ it follows  \begin{align*}
		[\ak g_{-1}\oplus \ak g_1,\ak g_{-1}\oplus \ak g_1]&= [\ak g_{-1},\ak g_{-1}]\oplus[\ak g_{-1},\ak g_{1}]\oplus[\ak g_{1},\ak g_{-1}]\oplus[\ak g_{1},\ak g_{1}].
	\end{align*}
	Hence $\ak{sl}_r(\sr A_R)^{\sigma^\pm}=\ak g_1=[\ak g_{-1},\ak g_{-1}]\oplus[\ak g_{1},\ak g_{1}].$ So the derivative of $\lambda$ at the identity is zero. Since $\lambda$ is a group homomorphism, its derivative  is identically zero everywhere. Since $\xb{SL}_r(\sr A_X)^{\sigma^\pm}$ is integral, we can write it as limit of integral varieties $V_n$ and for $n$ large $1\in V_n$, so $\lambda|_{V_n}=1$, hence $\lambda=1$.
\end{proof}

Fix an integer $k>0$. For any dominant weight $\lambda^\pm\in \x{P}^{\sigma,\pm}$,  there is a line bundle $\sr{L}(\lambda^\pm)$ over $\sr{Q}^{\sigma,\pm}$ associated to the principal $\xb{SL}_r(\sr O)^{\sigma^\pm}-$bundle:  $$\xb{SL}_r(\sr K)^{\sigma^\pm}\lra \sr{Q}^{\sigma,\pm},$$ defined using the character $e^{-\lambda^\pm}$ on $\xb{SL}_r(\sr O)^{\sigma^\pm}$.  Further, it is shown in \cite{Ku} that the space of global sections of powers of $\sr L(\lambda)$ is isomorphic to the dual of the irreducible  highest integrable representation of $\widehat{\al L}(\ak{sl}_r,\sigma^\pm)$ associated to $\lambda^\pm$.   \\ We are mainly interested in the case where $\lambda^\pm=\lambda_0^\pm$. Denote by $\al H_{\pm}(k)$ the highest weight representation of level $k$  of $\widehat{\al L}(\ak{sl}_r,\sigma^\pm)$ associated to the weight $\lambda_0^\pm$. It is called the basic representation of level $k$. So the above result of \cite{Ku} (see also \cite{Ma}) can be formulated as follows
\begin{theo}[Kumar, Mathieu] \label{Kum} 
	\begin{enumerate}
		\item The space  $H^0(\sr Q^{\sigma,-}, q^*\al P^k)$ is canonically isomorphic, as $\widehat{\al L}(\ak{sl}_r,\sigma^-)-$module, to the dual of the basic representation $\al H_-(k)$.
		\item The space  $H^0(\sr Q^{\sigma,+}, q^*\al D^k)$ is canonically isomorphic, as $\widehat{\al L}(\ak{sl}_r,\sigma^+)-$module, to the dual of the basic representation $\al H_+(k)$.
	\end{enumerate}
\end{theo}
Note that by Remark \ref{level}, when $r$ is even,  the weight $\lambda_0^+$ has level $2$ while $\lambda_0^-$ is of level $1$. This explains why we have to take the determinant line bundle in $\sigma^+$ case and the Pfaffian line bundle in $\sigma^-$ case. \\

The point that should be stressed here is that in \cite{Ku}, Kumar  has defined the ind-group $\xb{SL}_r(\sr O)\bs\xb{SL}_r(\sr K)$ using representation theory of Kac-Moody algebras. It is shown in \cite{LB} that this construction coincides with the usual functorial definition. Moreover, Pappas and Rapoport have claimed in \cite{PR} (page $3$) that the constructions of Kumar coincide with their definitions of the Schubert varieties. In particular, we deduce  in our  spacial case that the ind-variety structure on the twisted flag varieties  $\sr{Q}^{\sigma,\pm}$ are the same as those defined by Kumar. \\

Using the above results and assumptions, we deduce the following result which is a special case of   Conjecture $3.7$ of  Pappas and  Rapoport (\cite{PR}).
\begin{prop} \label{descent}
	We have isomorphisms $$H^0(\sr{SU}_X^{\sigma,-}(r),\al P^k)\cong \left(\prod_{p\in R}H^0(\sr Q_p^{\sigma,-},\al L_{-}^k)\right)^{\ak{sl}_r(\sr A_R)^{\sigma^-}},$$ $$H^0(\sr{SU}_X^{\sigma,+}(r),\al D^k)\cong\left( \prod_{p\in R}H^0(\sr Q_p^{\sigma,+},\al L^k)\right)^{\ak{sl}_r(\sr A_R)^{\sigma^+}}.$$
\end{prop}
\begin{proof} Since $\xb{SL}_r(\sr A_R)^{\sigma^\pm}$ and $\sr{Q}^{\sigma^\pm}$ are integral, the result follows, using the K\"unneth formula, from \cite[Proposition $7.4$]{BL}. 
\end{proof}

Now, Lemma \ref{descent} and Theorem \ref{Kum} imply  our main result
\begin{theo}Let $k\in \bb N$, we have 
	\begin{enumerate}
		\item  The space of global sections $H^0(\sr{SU}_X^{\sigma,-}(r),\al P^k)$ is canonically isomorphic to the conformal block space $\al V_{\sigma,-}(k)$.
		\item The space of global sections $H^0(\sr{SU}_X^{\sigma,+}(r),\al D^k)$ is canonically isomorphic to the conformal block space $\al V_{\sigma,+}(k)$.\\
	\end{enumerate}
\end{theo}

\section{Strange duality at level one} In this section, we  show a strange duality at level one for the moduli spaces $\al{SU}_X^{\sigma,+}(r)$ of $\sigma-$symmetric vector bundles in the unramiffied case. Since $\al{U}_X^{\sigma,-}(r)\cong \al{U}_X^{\sigma,+}(r)$, similar results hold for the $\sigma-$alternating case.

So assume that $\pi:X\ra Y$ is \'etale. Let $\Delta\in J_Y[2]$ the line bundle associated to this cover.  We use notation from subsection \ref{unramifiedPfaffian}. In particular, since we will deal just with the $\sigma-$symmetric case, we shall denote simply by $\x{P}^{ev}$ the space $\x{P}^{ev}_+$ and by $\al P_{ev}$ the isomorphism class of the Pfaffian bundles $\al P_L$, for $L\in \x{P}^{ev}$.

The moduli stack $\sr{U}_X^{\sigma,+}(r)$ of $\sigma-$symmetric bundles has two connected components distinguished by the parity of $h^0(E\otimes L)$, for fixed $L\in \x{P}^{ev}$. The even connected component is denoted $\sr U_{X,0}^{\sigma,+}(r)$. The moduli $\sr{SU}_X^{\sigma,+}(r)$ are connected. The associated moduli spaces has been constructed in \cite{Z2}. Here we consider the moduli spaces $\al{U}_{X}^{\sigma,+}(r)$ (resp. $\al{SU}_X^{\sigma,+}(r)$) of \emph{stable} $\sigma-$symmetric vector bundles (resp. with trivial determinant). 
\begin{lemm}\label{descendpfaffian}
	The Pfaffian line bundle $\al P_{ev}$ over $\sr{U}_{X,0}^{\sigma,+}(r)$ descends to the moduli space $\al{U}_{X,0}^{\sigma,+}(r)$.
\end{lemm}
\begin{proof}
	The moduli space $\al{U}_X^{\sigma,+}(r)$ is constructed using GIT as a $\x{SL}(H)-$quotient of a parameter scheme $\x{Quot}^{\sigma}(\bb C)$, where $H=\bb C^{m}$ for some $m$ (see \cite{Z2}). Let $L\in \x{P}^{ev}$ 
	 and  $a=(E,q,\overline{\psi})$ be a point of $\x{Quot}^\sigma(\bb C)$. Since $E$ is stable,  the stabilizer of $a$ under the action of $\x{SL}(H)$ is just $\{\pm1\}$. The action of this stabilizer on $(\al P_L)_a$ is by definition multiplication by $g^{h^1(E\otimes L)}$, for $g\in\{\pm1\}$. Since $\al{U}_{X,0}^{\sigma,+}(r)$ is connected, we have $$h^1(E\otimes L)=\begin{cases}1& \x{if}\;r\equiv 1\mod2\x{  and }L\in\x{P}^{od} \\ 0&\x{otherwise}.\end{cases}$$
	This can be shown using Hitchin system (cf. \cite[Theorem $4.12$]{Z}). 
	Since $L$ is even, it follows that $-1$ acts trivially on $(\al P_L)_a$, for any $a$.  Using Kempf's Lemma we deduce the result.
\end{proof}

Now we show the existence of the Pfaffian divisor. Let  $\al{U}_X(r,0)$ be  the moduli space  of rank $r$ and  degree $0$ stable vector bundles over $X$, and let $\Theta_L$ be the divisor in $\al U_X(r,0)$  supported on vector bundles $E$ such that $E\otimes L$ has a non-zero global section, where $L\in \x{P}^{ev}$ is fixed. 

Let us recall from \cite{Z} some basic results about the Hitchin system in this case. The Hitchin morphism on $\al{U}_X(r,0)$ induces a fibration $$\sr H:T^*\al U_X^{\sigma,+}(r)\lra W^{\sigma,+},$$ where $W^{\sigma,+}=\bigoplus H^0(K_X^i)_+$. For general $s\in W^{\sigma,+}$ the associated spectral curve $q:\tilde{X}_s\ra X$ is smooth and it has a fixed point free involution $\tilde{\sigma}$ that lifts $\sigma$. Moreover, the quotient curve $\tilde{Y}_s=\tilde{X}_s/\tilde{\sigma}$ is a smooth spectral curve over $Y$ with spectral data in $K_Y\otimes \Delta$. Let $S$ be the ramification divisor of $\tilde{Y}_s\ra Y$.  Then the fiber $\x{Nm}_{\tilde{X}/\tilde{Y}_s}^{-1}(\al{O}(S))$ has two connected components $\tilde{\x{P}}^{ev}\cup \tilde{\x{P}}^{od}$, distinguished by the parity of  $h^0(-\otimes q^*L)$ ($+$ for even), where $L$ is in $\x{P}^{ev}$.  Now by \cite[Theorem $4.17$]{Z} the push-forward map $$q_*:\tilde{\x{P}}^{ev}\cap \tilde{ \x{Q}}\dashrightarrow \al{SU}_X^{\sigma,+}(r)$$ is dominant, where $\tilde{ \x{Q}}=\x{Nm}_{\tilde{X}_s/X}^{-1}(\delta)$, $\delta=\x{det}(q_*\al O_{\tilde{X}_s})^{-1}$.

\begin{lemm}\label{effective}
	Let $L\in{ \x{P}}^{ev}$. The  restrictions of the divisor $\Theta_{L}\subset \al U_X(r,0)$ to $\al{U}_{X,0}^{\sigma,+}(r)$ and $\al{SU}_X^{\sigma,+}(r)$ are again divisors. The associated reduced  divisors are denoted $\Xi_L$.
\end{lemm}
\begin{proof}
	It is enough to produce a semistable $\sigma-$symmetric vector bundle with trivial determinant that does not belong to $\Theta_L$.	  Let $\Xi_L$ be the (principal) polarization on $\x{P}:=Prym(X\ra Y)$. Then the linear system $|r\Xi_L|$ is base point free for any $r\geqslant2$ and the group $\x{P}[r]$ acts irreducibly on it. Hence let $\alpha\in \x{P}[r]$ such that $\al O_X\not\in T_\alpha^*\Xi_L=\Xi_{\alpha\otimes L}$. In other words $h^0(\alpha\otimes L)=0$. Define $E:=\alpha^{\oplus r}$. It is obviously a semistable $\sigma-$symmetric  vector bundle with trivial determinant, hence it belongs to the closures of $\al{SU}_X^{\sigma,\pm}(r)$ and $\al U_{X,0}^{\sigma,+}(r)$  and it is not in the restriction of the divisor $\Theta_L$. 
\end{proof}

Note that the other connected component $\al{U}_{X,1}^{\sigma,+}(r)$ is entirely included in $\Theta_{L}$ for any $L\in\x{P}^{ev}$. For the moduli of $\sigma-$alternating bundles, the same happens, i.e. the restriction of $\Theta_{L}$ to $\al{U}_{X,0}^{\sigma,-}(r)$ is again a divisor and $\al{U}_{X,1}^{\sigma,-}(r)\subset \Theta_L$. \\

\begin{lemm}
We have 	$\x{dim}(H^0(\al{U}_{X,0}^{\sigma,+}(r),\al{P}_{ev}))=1.$
\end{lemm}
\begin{proof}
	Let $q:\tilde{X}_s\ra X$ be a smooth spectral curve over $X$ attached to a general $s\in W^{\sigma,+}$ (see \cite{Z} for more details and notations).  First, for some positive integer $m$, the pullback of the determinant bundle via $q_*:J_{\tilde{X}_s}^m\ra \al{U}_X(r,0)$ is the line bundle $\al O(\Theta_{q^*\kappa})$ attached to the Riemann theta divisor $\Theta_{q^*\kappa}$ over $J_{\tilde{X}_s}^m$ (modulo a translation).  Let $\al S\subset \tilde{\x{P}}^{ev}$ be the locus of line bundles $L$ such that $q_*L$ is stable. One can show (by restricting every thing from \cite{BNR}) that the codimension of the complement of $\al S$ is at least $2$. Since $q_*:\al 
	S\lra \al{U}_{X,0}^{\sigma,+}(r)$ is dominant, we get an injection $$H^0(\al{U}_{X,0}^{\sigma,+}(r),\al P_L)\hookrightarrow H^0(\tilde{\x{P}}^{ev}, \tilde{\al L}),$$ where $\tilde{\al L}$ is the principle polarization on $\tilde{\x{P}}^{ev}$. So $h^0(\al{U}_{X,0}^{\sigma,+}(r),\al P_L)$ is at most $1$. \\
		Now by Lemma \ref{effective}, there is an effective divisor $\Xi_L$ such that $2\Xi_L=\Theta_L|_{\al U_{X,0}^{\sigma,+}(r)}$.  In particular,  $\al{P}_L$ has a non trivial global section.
	%
\end{proof}

Denote by $\al L$ and $\tilde{\al L}$ line bundles defining principal polarizations on $\x{P}^{ev}$ and $\tilde{\x{P}}^{ev}$ respectively. 
    The restriction of  $\al P_{ev}$ to $\al{SU}_X^{\sigma,+}(r)$ is denoted again by $\al P_{ev}$.
\begin{theo}\label{PSD}
	We have an isomorphism $$ H^0(\x{P}^{ev},\al L^r)^*\cong H^0(\al{SU}_X^{\sigma,+}(r),\al{P}_L).$$ In particular we deduce $$\x{dim}(H^0(\al{SU}_X^{\sigma,+}(r),\al P_L))=r^{g_Y-1}.$$ 
\end{theo}
\begin{proof}
	Consider the following commutative diagram $$\xymatrix{\tilde{\x{P}}^{ev}\cap \tilde{ \x{Q}} \times \x{P}^{ev}\ar[rr]\ar[d] & & \tilde{\x{P}}^{ev}\ar[d]\\ \al{SU}_X^{\sigma,+}(r)\times \x{P}^{ev}\ar[rr] &&  \al U_{X,0}^{\sigma,+}(r,K_X), }$$
	where $\al U_{X,0}^{\sigma,+}(r,K_X)=\{E\otimes L\;|\;E\in\al U_{X,0}^{\sigma,+}(r)\}$, it has a canonical Pfaffian line bundle $\al P$. 	Using  \cite[ Theorem $3$]{BNR}, we deduce that  the pullback of the line bundle $\al P$ to $\al{SU}_X^{\sigma,+}(r)\times \x{P}^{ev}$ is isomorphic to  $p_1^*\al P_{ev}\otimes p_2^*\al L^r$.\\ 
	Now the rational map $\tilde{\x{P}}^{ev}\cap \tilde{ \x{Q}} \lra \al{SU}_X^{\sigma,+}(r)$ is dominant (\cite[Theorem $4.16$]{Z}). It follows, by the same argument used in the proof above, that the map $$H^0(\al{SU}_X^{\sigma,+}(r), \al P_{ev})\rightarrow H^0(\tilde{\x{P}}^{ev}\cap  \tilde{ \x{Q}}, \tilde{\al L})$$ is injective,	where here we denote abusively by $\tilde{\al L}$ the restriction of $\tilde{\al L}$ to $\tilde{P}^{ev}\cap \tilde{ \x{Q}}\subset\tilde\tilde{ \x{P}}^{ev}$. Since the two subvarieties $\x{P}^{ev}$ and $\tilde{\x{P}}^{ev}\cap \tilde{ \x{Q}}$ are (torsors under) complementary pair inside $\tilde{\x{P}}^{ev}$, we obtain,  using \cite[Proposition $2.4$]{BNR},  an isomorphism $$H^0(\tilde{\x{P}}^{ev}\cap \tilde{ \x{Q}}, \tilde{\al L})\cong H^0(\x{P}^{ev},\al L^r)^*.$$
	Hence we deduce an injective map $$H^0(\al{SU}_X^{\sigma,+}(r),\al P_{ev})\hookrightarrow H^0(\x{P}^{ev},\al L^r)^*.$$
	Moreover the group $\x{P}[r]$ acts on  $\bb{P}H^0(\x{P}^{ev},\al L^r)^*$ as well as on $\al{SU}_X^{\sigma,+}(r)$, hence it acts also on  the linear system $\bb PH^0(\al{SU}_X^{\sigma,+}(r),\al P_{ev})$. Since the projective  representation $\bb PH^0(\x{P}^{ev},\al L^r)^*$ is irreducible, the  map $H^0(\al{SU}_X^{\sigma,+}(r),\al P_{ev})\hookrightarrow H^0(\x{P}^{ev},\al L^r)^*$, which is equivariant for these actions, is necessarily an isomorphism.
\end{proof}
Note that we have a map $$\rho:\al{SU}_X^{\sigma,+}(r)\dashrightarrow |\al{L}^r|=\bb PH^0(\x{P}^{ev},\al{L}^r),$$ given by $$E\lra \rho(E)=\x{divPf}(\pi_*(p_2E\otimes\sr K)),$$
where  $\sr K$ is the normalized  Poincar\'e bundle over $\x{P}^{ev}$ such that $\sigma^*\sr K\simeq\sr  K^{-1}\otimes p_2^*K_X$. Note that the family $\pi_*(p_2^*E\otimes \sr K)$ has a non-degenerated quadratic form with values in $K_Y$. Note that $\rho^*\al O(1)\cong \al P_{ev}$. So this map and the duality of complimentary pairs induce the following commutative diagram $$\xymatrix{H^0(\x{P}^{ev},\al{L}^r)^*\ar[rr]\ar[rrd]&& H^0(\al{SU}_X^{\sigma,+}(r),\al P_{ev})\ar[d] \\ & &H^0(\tilde{ \x{P}}^{ev}\cap \tilde{ \x{Q}},\tilde{\al L}),  }$$ 
where all maps in the above diagram are isomorphisms. In other words, the isomorphism of Theorem \ref{PSD} is exactly $\rho^*$. 
\vspace{1cm}

\bibliographystyle{alpha}
\bibliography{bib}
\end{document}